\documentclass[10pt]{amsart}
\usepackage{geometry} 
\usepackage{hyperref}
\newtheorem{defi}{Definition}[subsection]
\newtheorem*{defi*}{Definition}
\newtheorem{rmk}[defi]{Remark}
\newtheorem*{rmk*}{Remark}
\newtheorem{prop}[defi]{Proposition}
\newtheorem*{prop*}{Proposition}
\newtheorem{thm}[defi]{Theorem}
\newtheorem*{thm*}{Theorem}
\newtheorem{cor}[defi]{Corollary}
\newtheorem*{lmm*}{Lemma}
\newtheorem{lmm}[defi]{Lemma}
\newtheorem{ex}[defi]{Example}

\title{On a theory of the $b$-function in positive characteristic}
\author{Thomas Bitoun}
\address{Mathematical Institute, University of Oxford, Oxford OX2 6GG, UK}
\email{tbitoun@gmail.com}
\date{} 

\begin{document}

\newtheorem{MainThm}{Theorem} 
\renewcommand{\theMainThm}{\Alph{MainThm}}. 

\maketitle

\begin{abstract} 
We present a theory of the $b$-function (or Bernstein-Sato polynomial) in positive characteristic. 
Let $f$ be a non-constant polynomial with coefficients in a perfect field $k$ of characteristic $p>0.$ Its $b$-function $b_f$ is defined to be an ideal of the algebra of continuous $k$-valued functions on $\mathbb{Z}_p.$ The zero-locus of the $b$-function is thus naturally interpreted as a subset of $\mathbb{Z}_p,$ which we call the set of roots of $b_f.$ We prove that $b_f$ has finitely many roots and that they are negative rational numbers. 
Our construction builds on an earlier work of Musta{\c{t}}{\u{a}} and is in terms of $D$-modules, where $D$ is the ring of Grothendieck differential operators. We use the Frobenius to obtain finiteness properties of $b_f$ and relate it to the test ideals of $f.$ 
\end{abstract}

\tableofcontents
\addtocontents{toc}{\protect\setcounter{tocdepth}{2}}

\section*{Introduction}

\subsection*{Bernstein-Sato polynomial}

Let $f$ be a non-constant complex polynomial in $n$ variables. Let us consider after Bernstein the functional equation in $P(s)$ and $q(s)$: \begin{equation} \label{equation: Bernstein}
P(s)f^{s+1}= q(s)f^s, \tag{$\star$}
\end{equation} 
where $P(s)$ is an $s$-polynomial with coefficients in differential operators with polynomial coefficients in $n$ variables and $q(s)$ is a polynomial of the variable $s.$ The fundamental result is that there is always a solution $(P(s), q(s))$ of the functional equation such that the polynomial $q(s)$ is non-zero (\cite[Theorem 1']{MR0320735}). Since moreover the second components $q(s)$ of the solutions obviously form an ideal of $\mathbb{C}[s],$ one may consider the monic generator $b_f(s)$ of that ideal. It is called the Bernstein-Sato polynomial (or the $b$-function) of $f.$ 

The roots of the Bernstein-Sato polynomial are not arbitrary. Indeed Kashiwara showed the following (\cite[Corollary (5.2)]{MR0430304}):
\begin{MainThm} \label{MainThm: Kashiwara} The roots of $b_f(s)$ are strictly negative rational numbers.
\end{MainThm}

The consideration of the solutions to the functional equation above was originally motivated by the problem of analytically continuing the zeta function attached to $f,$ see \cite[Theorem 1]{MR0320735}. We would like to highlight two other types of consequences of that result.  

On one hand the existence of the Bernstein-Sato polynomial of a general non-constant polynomial is of theoretical interest for modules over rings of differential operators or $\mathcal{D}$-modules. Namely it was the latter's first application and has been crucial in singling out the importance of holonomicity (\cite[Corollary 1.4.b)]{MR0320735}). Moreover it is used to prove the stability of holonomicity under operations, e.g. direct image for an open embedding (\cite[\href{http://www.math.columbia.edu/~khovanov/resources/Bernstein-dmod.pdf}{p.25}]{BernsteinLN}), nearby and vanishing cycles (\cite[2.1. Key Lemma]{MR923134}).

On the other hand, the Bernstein-Sato polynomial of $f$ is a rich invariant of the singularity $\{f=0\},$ as was discovered by Malgrange (\cite{MR737934}). In particular the roots of $b_f(s)$ are related to very many invariants of the singularities of $f,$ see \cite{saito2006introduction}, \cite{MR1492525} and \cite{leykin2015survey} 
for surveys. Of motivational interest for us is 
the following result of Ein-Lazarsfeld-Smith-Varolin (\cite[Theorem B]{MR2068967}), relating the jumping coefficients of the multiplier ideals of $f$ to the roots of $b_f(s),$ see \cite{MR2068967} for the definitions:

\begin{MainThm} \label{MainThm: ELSV} Let $\lambda$ be a jumping coefficient of $f$ which lies in $(0,1].$ Then $-\lambda$ is a root of the Bernstein-Sato polynomial of $f.$
\end{MainThm}

In this paper, we put forth a natural definition of Bernstein-Sato polynomial (or $b$-function) in positive characteristic and show that it satisfies statements analogous to Theorems \ref{MainThm: Kashiwara} and \ref{MainThm: ELSV}.  

\subsection*{Positive characteristic} Let $k$ be a perfect field of positive characteristic $p$ and let $g$ be a non-constant $k$-polynomial in $n$ variables. 
A positive characteristic analogue of the multiplier ideals is provided by the (generalised) test ideals of Hara and Yoshida, see \cite{MR1974679}. 
Where the $F$-jumping exponents of the test ideals of $g$ are the pendant to the 
jumping coefficients of the multiplier ideals. 
These are rational numbers (\cite[Theorem 3.1]{MR2492440}). 

In accordance with the problematic set forth by Musta{\c{t}}{\u{a}} in \cite{MR2469353} and in analogy with Theorem \ref{MainThm: ELSV}, we would like the sought for $b$-function of $g$ to be related to its $F$-jumping exponents. It does thus seem preferable 
for the $b$-function to be an object living beyond the positive characteristic. 

The theory we propose here is $p$-adic, in the sense that the roots of our $b$-function are naturally $p$-adic integers. To describe it, let us first consider the usual abstraction of Bernstein's functional equation, see e.g. \cite[\S1]{MR0430304}.   

Let $\mathcal{D}$ be the ring of complex differential operators with polynomial coefficients in $n$ variables, and let $\mathcal{O}$ be the ring of complex polynomials in $n$ variables. We use the notation $\mathcal{D}[s]$ for the algebra $\mathcal{D}\otimes_{\mathbb{C}}\mathbb{C}[s],$ where the variable $s$ is central. For a non-constant element $f$ of $\mathcal{O},$ let $\mathcal{O}[\frac{1}{f}][s]f^s$ be the free $\mathcal{O}[\frac{1}{f}][s]$-module of rank one with generator $f^s.$ 
It is endowed with a compatible action of $\mathcal{D}[s]$  
by making a derivation $\partial$ in $\mathcal{D}$ 
act on the generator $f^s$ by $\partial \cdot f^s:= s \partial(f)f^{-1}f^s.$ The left $\mathcal{D}[s]$-module $\mathcal{D}[s]f^s$ is the $\mathcal{D}[s]$-submodule of $\mathcal{O}[\frac{1}{f}][s]f^s$ generated by $f^s$ and the one 
generated by $f\cdot f^s$ is denoted $\mathcal{D}[s]f^{s+1}.$ 
It is clear that there is a differential operator $P(s)\in \mathcal{D}[s]$ such that $(P(s), q(s))$ is a solution to Bernstein's functional equation (\ref{equation: Bernstein}) if and only if the polynomial $q(s)$ annihilates the $\mathcal{D}[s]$-module $\mathcal{D}[s]f^s/\mathcal{D}[s]f^{s+1}.$ Thus the Bernstein-Sato polynomial of $f$ is the monic generator of the annihilator of the action of $\mathbb{C}[s]$ on $\mathcal{D}[s]f^s/\mathcal{D}[s]f^{s+1}.$

In the positive characteristic theory, the r\^ole of $\mathbb{C}[s]$ is played by the $k$-algebra $A_k$ generated by the binomial coefficients functions modulo $p.$ It is not principal but there is nevertheless a natural notion of roots of an ideal. Indeed $A_k$ is isomorphic to $k[Y_e; e\in\mathbb{N}]/ (Y_e^p-Y_e; e\in\mathbb{N})$ and the generators are canonically ordered. Thus for each homomorphism $\alpha$ from $A_k$ to a field, $\alpha(Y_e)$ is in the prime field, for all $e.$ We associate to $\alpha$ the $p$-adic integer $\Sigma_e\alpha_ep^e,$ where $\alpha_e$ is the unique lift of
$\alpha(Y_e)$ to $\{0, \dots, p-1\}.$ In particular to each maximal ideal of $A_k,$ we may attach a $p$-adic integer for $\alpha$ the canonical quotient homomorphism. In fact, we show that $A_k$ is isomorphic to the $k$-algebra $\Lambda_k$ of continuous $k$-valued functions on the $p$-adic integers. Thus to each ideal $I$ of $A_k$ corresponds a set of $p$-adic integers. Namely those associated by the above to the maximal ideals containing $I.$ We call them the roots of $I.$ 

Our approach to the $b$-function of a non-constant polynomial $g$ in positive characteristic is to consider a left $D\otimes_kA_k$-module $N_g^\gamma$ for $D$ the ring of Grothendieck differential operators, which is thought of as the positive characteristic analogue of $\mathcal{D}[s]f^s/\mathcal{D}[s]f^{s+1}$ (for $g$ instead of $f$), see Definition \ref{defi: N_f}. We then define the $b$-function $b_g$ of $g$ to be the annihilator of the action of $A_k$ on $N_g^\gamma.$ The roots of $b_g$ satisfy strong finiteness properties. Indeed, here are those corresponding to Theorems \ref{MainThm: Kashiwara} and \ref{MainThm: ELSV}:  

\begin{MainThm} The roots of the $b$-function of $g$ form a finite set of strictly negative rational numbers. Furthermore they are not smaller than $-1.$ 
\end{MainThm}

In fact, we have the following precise relation to the $F$-jumping exponents of $g,$ see Theorem \ref{thm:periodic expansion of roots}.

\begin{MainThm} Let $\lambda$ be an $F$-jumping exponent of $g$ in $(0,1]\cap\mathbb{Z}_{(p)}.$ Then $-\lambda$ is a root of the $b$-function of $g.$ Moreover this exhausts the roots of $b_g.$
\end{MainThm}

In order to prove our main results, we study in some depth the $D$-module structure of $N_g^\gamma.$ We show that it is a finitely generated unit $F$-module (Theorem \ref{thm:structure N_f}), which serves as a replacement for holonomicity. In fact we exhibit an explicit generator of the unit $F$-module $N_g^\gamma$ expressed in terms of the test ideals of $g.$ This is what ultimately allows us to relate the roots of the $b$-function to the $F$-jumping exponents and prove Theorem \ref{thm:periodic expansion of roots}. 

\subsection*{Further comments}

\subsubsection*{Relation to the work of Musta{\c{t}}{\u{a}} and motivation}

We would like to briefly comment on the relationship between this work and \cite{MR2469353}. In \cite{MR2469353}, which was our starting point, Musta{\c{t}}{\u{a}} uses the action of Euler operators on bounded level analogues of $N_g$ to construct a whole sequence
of Bernstein-Sato-type polynomials and his main result is that the information provided by these invariants corresponds to that provided by the $F$-jumping exponents of $g.$ However, this correspondence leaves open the question of what is the natural analogue of the Bernstein-Sato polynomial. This is the question addressed here. We do so by focusing on $D$-modules (as opposed to $D^{(e)}$-modules with bounded divided powers) and hence on the whole algebra of higher Euler operators. This leads us to a new notion of $b$-function (or Bernstein-Sato polynomial), naturally $p$-adic. 

\subsubsection*{Nearby cycles} 

We have generalised our theory to unit $F$-modules coefficients and argue that these may be called nearby cycles. This will appear elsewhere.

\subsubsection*{Frobenius structure} The action of Frobenius on $N_g$ is given here via an explicit root for the corresponding unit $F$-structure. It would be preferable to have a direct description. We have failed to obtain one so far.

\subsubsection*{Relation to the Bernstein-Sato polynomial} 
In view of the analogy between the Bernstein-Sato polynomial of a, say, complex polynomial and the theory of the $b$-function presented here, it is natural to ask about the comparison between the Bernstein-Sato polynomial of a polynomial $f$ with say rational coefficients and the $b$-function of its reduction $f_p$ modulo a prime $p,$ for $p$ large. 

The situation is subtle, as shown in Section \ref{section: ex}. Strikingly, there are polynomials $f$ such that there are arbitrary large primes $p$ for which there is a root of the $b$-function of $f_p$ which is not a root of the Bernstein-Sato polynomial of $f,$ see Example \ref{ex: Takagi}. 

\subsection*{Acknowledgements}

This theory of the $b$-function in characteristic $p$ sprang from my attempt to understand \cite{MR2469353}. I am indebted to Mircea Musta{\c{t}}{\u{a}} for the work done there. I would also like to thank him for initially mentioning the problem and answering many questions about test ideals. I thank Konstantin Ardakov and Francesco Baldassarri for introducing me to Mahler's Theorem, thus clarifying the appearance of $\mathbb{Z}_p$ in the theory. Shunsuke Takagi kindly pointed me towards Example \ref{ex: Takagi}. Finally I would like to thank Roman Bezrukavnikov and Pavel Etingof for interesting discussions at the very beginning of this project. 
The author was partially supported by EPSRC grant EP/L005190/1.

\section{Construction} \label{section: construction}
\subsection{The algebra of higher Euler operators} 

Let $k$ be a field of characteristic $p>0$ and let $R$ be a commutative $k$-algebra. In the sequel, we will use the notation $D_R$ for the ring of global sections of the sheaf of Grothendieck differential operators $D_{Spec(R)/k}.$
We refer to \cite[16.8]{MR0238860} for a general treatment of the sheaf of rings of Grothendieck differential operators. 

After \cite{MR2469353}, we set the \textit{higher Euler operators} to be the following differential operators in $D_{k[t]}:$ 

\begin{defi} Let $e\geq0$ be a natural number. The \emph{$e$-th higher Euler operator} is 
$$\nu_e:=\frac{d}{dt}^{[p^e]}t^{p^e},$$ where $\frac{d}{dt}^{[p^e]}$ is the divided power differential operator such that for all $n\geq0,$ $\frac{d}{dt}^{[p^e]}(t^n)= {n\choose p^e}t^{n-p^e}.$ (We set ${n \choose m}=0$ for $n<m.$)
\end{defi}

\begin{defi} The \emph{$k$-algebra of higher Euler operators} $\Gamma_k$ is the unital $k$-subalgebra of $D_{k[t]}$ generated by the higher Euler operators. 
\end{defi}

For all $n\geq0,$ let ${s \choose n}$ denote the $n$-th binomial coefficient function modulo $p,$ $$\mathbb{N}\to \mathbb{F}_p, l\mapsto {l \choose n}.$$ We will consider them as $k$-valued functions.

\begin{defi} The \emph{$k$-algebra of binomial coefficients} $k[{s \choose p^0}, {s \choose p^1}, {s \choose p^2}, \cdots]$ is the $k$-subalgebra of the algebra of $k$-valued functions on $\mathbb{N}$ generated by the binomial coefficients functions $\{{s \choose n}; n\in \mathbb{N}\}.$\end{defi}

Let us recall the classical theorem of Lucas (\cite{MR1483922}).

\begin{thm} \label{thm: Lucas} Let $m$ and $n$ be natural numbers and let $p$ be a prime number. The binomial coefficient ${m \choose n}$ modulo $p$ is given by:
$${m \choose n}= \Pi_{e=0}^{\infty}{m_e \choose n_e} \text{ mod }p,$$ where $m= \Sigma_{e=0}^{\infty}m_ep^e$ (resp. $n= \Sigma_{e=0}^{\infty}n_ep^e$) is the base $p$ expansion of $m$ (resp. $n$).
(We set ${0 \choose 0}=1.$)
\end{thm}

\begin{cor} \label{cor:presentation} The $k$-algebra of binomial coefficients is generated by the binomial coefficients $\{{s \choose p^e}; e\in \mathbb{N}\}.$ This justifies our choice of notation. Moreover, the kernel of the surjective morphism of $k$-algebras $\pi_k:$ $$k[Y_e; e\geq0] \to k[{s \choose p^0}, {s \choose p^1}, {s \choose p^2}, \cdots], Y_e\mapsto {s\choose p^e}$$ is the ideal generated by $\{Y_e^p-Y_e; e\in \mathbb{N}\}.$

\end{cor} 

\begin{proof} Let $n$ be a natural number and let $n= \Sigma_{e=0}^{N}n_ep^e$ be its base $p$ expansion. It follows directly from Lucas' Theorem that $${s \choose n}= \Pi_{e=0}^{N} {{s \choose p^e} \choose n_e}.$$ This proves that the $k$-algebra of binomial coefficients is generated by the binomial coefficients $\{{s \choose p^e}; e\in \mathbb{N}\}.$ 

Since the ${s \choose p^e}$ are $\mathbb{F}_p$-valued functions, it is clear that they satisfy the relations ${s \choose p^e}^p={s \choose p^e}.$ Thus $\{Y_e^p-Y_e; e\in \mathbb{N}\}$ is contained in the kernel of $\pi_k.$ Let us prove that the kernel is generated by $\{Y_e^p-Y_e; e\in \mathbb{N}\}.$
Since $\pi_k=k\otimes_{\mathbb{F}_p}\pi_{\mathbb{F}_p}$ and $k$ is flat over $\mathbb{F}_p,$ it is enough to show the assertion for $\pi_{\mathbb{F}_p}.$ 

Let $P$ be in the kernel. That is $P$ is a $\mathbb{F}_p$-polynomial in $N$ variables for a natural number $N$ and $P({s \choose p^0}, \dots, {s \choose p^{N-1}})=0.$ Thus by Lucas' Theorem, $P$ is in the kernel of the evaluation morphism: $$\mathbb{F}_p[Y_0, \dots, Y_{N-1}] \to Fun(\mathbb{F}_p^N, \mathbb{F}_p),$$ where $Fun(\mathbb{F}_p^N, \mathbb{F}_p)$ is the algebra of $\mathbb{F}_p$-valued functions on $\mathbb{F}_p^N.$ One easily sees by counting dimensions that this implies that $P$ lies in the ideal generated by $\{Y_e^p-Y_e;0\leq e\leq N-1\}.$ This concludes the proof of the corollary.\qedhere

\end{proof}

The $k$-algebra of higher Euler operators and the $k$-algebra of binomial coefficients are isomorphic. Let us fix an isomorphism. 

\begin{lmm} \label{lmm: gamma}The assignment ${s \choose p^e}\mapsto -\nu_e$ induces a morphism $\gamma$ of $k$-algebras: 
$$k[{s \choose p^0}, {s \choose p^1}, {s \choose p^2}, \cdots] \overset{\gamma}{\to}\Gamma_k.$$ The morphism $\gamma$ is an isomorphism.
\end{lmm}

\begin{proof} Let us first note that the definition of the higher Euler operators provides a natural embedding of $\Gamma_k$ into the algebra of binomial coefficients. Indeed, since $D_{\mathbb{A}_k^1}$ acts faithfully on $k[t],$ the higher Euler operators are characterised by their action on the monomials. It is given by $\nu_e(t^n)=\frac{d}{dt}^{[p^e]}t^{p^e}(t^n)={p^e+n \choose p^e}t^n= (1+{n \choose p^e})t^n,$ the last equality holding by Lucas' Theorem. We thus see that the assignment $\nu_e \mapsto 1+{s \choose p^e}$ defines an embedding of $\Gamma_k$ into the $k$-algebra of binomial coefficients. It clearly is an isomorphism.

We are thus left with having to show that the assignment ${s \choose p^e}\mapsto -1-{s \choose p^e}$ induces an automorphism of $k[{s \choose p^0}, {s \choose p^1}, {s \choose p^2}, \cdots].$ This is obvious from the presentation given in Corollary \ref{cor:presentation}.\qedhere
\end{proof}

\begin{rmk} The isomorphism $\gamma$ is the analogue of the isomorphism $$\mathbb{C}[s]\to \mathbb{C}[\frac{d}{dt}t], s\mapsto -\frac{d}{dt}t$$ from the complex theory.
\end{rmk}

Next we explicitly relate the algebra of binomial coefficients to the $p$-adic integers.

\begin{thm} \label{thm:p-adic roots} Let $\mathbb{F}_p$ and $k$ be endowed with the discrete topology. 
\begin{enumerate}
\item \label{item: bincontinuous}
For all natural numbers $e,$ the binomial coefficient function ${s \choose p^e}$ modulo $p$ extends to a continuous $\mathbb{F}_p$-valued function $c_e$ on $\mathbb{Z}_p,$ such that $c_e(z)=z_e \text{ mod } p,$ where $z=\Sigma_{e\geq0}z_ep^e$ is the $p$-adic expansion. This induces an embedding $i$ of $k$-algebras: $$k[{s \choose p^0}, {s \choose p^1}, {s \choose p^2}, \cdots] \overset{i}{\to} C(\mathbb{Z}_p, k),$$ where $C(\mathbb{Z}_p, k)$ is the $k$-algebra of continuous $k$-valued functions on $\mathbb{Z}_p.$ The morphism $i$ is an isomorphism.

\item \label{item: maxcontinuous}
Moreover, for each $p$-adic integer $z,$ let $ev_z$ be the evaluation morphism: $C(\mathbb{Z}_p, k) \to k, f\mapsto f(z).$ The map $\epsilon$ from the $p$-adic integers to the set of maximal ideals $Max(C(\mathbb{Z}_p, k))$ of $C(\mathbb{Z}_p, k)$ given by $$\mathbb{Z}_p \overset{\epsilon}{\to} Max(C(\mathbb{Z}_p, k)), z \mapsto ker(ev_z)$$ is a bijection. 
\end{enumerate}
\end{thm}

\begin{proof} 
(\ref{item: bincontinuous}) The projection to the $e$-th coefficient in the $p$-adic expansion is a continuous map from $\mathbb{Z}_p$ to $\mathbb{F}_p.$ It is thus continuous as a function from $\mathbb{Z}_p$ to $k.$ Moreover it coincides with ${s \choose p^e}$ on $\mathbb{N}$ by Lucas' Theorem (Theorem \ref{thm: Lucas}). This provides the morphism $i$ of the statement. It is injective since $\mathbb{N}$ is dense in $\mathbb{Z}_p.$ As the ${s \choose p^e}$ generate the algebra of binomial coefficients by Corollary \ref{cor:presentation}, it is a direct consequence of Mahler's Theorem, see e.g. \cite[III 1.2.4]{MR0209286}, that $i$ is surjective. 

(\ref{item: maxcontinuous}) Since ${s \choose p^e}$ is the projection to the $e$-th coordinate in the $p$-adic expansion, it is clear that the map $\mathbb{Z}_p \overset{\epsilon}{\to} Max(C(\mathbb{Z}_p, k))$ is injective. Indeed if $z=\Sigma_{e\geq0}z_ep^e$ is different from $z'=\Sigma_{e\geq0}z_e'p^e,$ then there is a number $l\geq0$ such that $z_l$ is different from $z_l'.$ Thus ${s \choose p^l} - z_l$ belongs to $ker(ev_z)$ but not to $ker(ev_{z'}).$

Let us show that $\epsilon$ is surjective. 
Let $z=\Sigma_{e\geq0}z_ep^e,$ the maximal ideal $ker(ev_z)$ is generated by $\{{s \choose p^e} - z_e; e\geq 0\}.$ Indeed the set $\{{s \choose p^e} - z_e; e\geq 0\}$ is contained in $ker(ev_z),$ and since the ideal generated by $\{{s \choose p^e} - z_e; e\geq 0\}$ is obviously maximal, it has to be equal to $ker(ev_z).$
Let $\mathfrak{m}$ be a maximal ideal of $C(\mathbb{Z}_p, k)$ and let $$C(\mathbb{Z}_p, k) \overset{m}{\to} K:= C(\mathbb{Z}_p, k)/\mathfrak{m}$$ be the quotient morphism. Since ${s \choose p^e}^p={s \choose p^e}, m({s \choose p^e})$ is in the prime field $\mathbb{F}_p,$ for all $e.$ Let $a_e$ be the lift of $m({s \choose p^e})$ to $\{0, \dots, p-1\}$ and let $a=\Sigma_{e\geq0}a_ep^e.$ The maximal ideal $\mathfrak{m}$ contains $ker(ev_a).$ They are thus equal. This shows that $\epsilon$ is surjective and concludes the proof of the theorem.\qedhere
\end{proof} 

Let us single out a special word for the vanishing locus of an ideal. 

\begin{defi} 

Let $I$ be an ideal of the $k$-algebra of binomial coefficients. The \emph{roots} of $I$ are the $p$-adic integers corresponding by Theorem \ref{thm:p-adic roots} to the maximal ideals containing $I.$

\end{defi}

The following proposition shows that the naive notion of multiplicity is not so interesting for the algebra of binomial coefficients. 

\begin{prop} \label{prop: multiplicity}
\begin{enumerate}

\item \label{item: Jacobson}
The $k$-algebra of binomial coefficients is a Jacobson ring.
\item \label{item: radical}
The ideals of $\mathbb{F}_p[{s \choose p^0}, {s \choose p^1}, {s \choose p^2}, \cdots]$ are radical. They are thus characterised by their sets of roots. 
\end{enumerate}
\end{prop}

\begin{proof} 
(\ref{item: Jacobson}) Let $A$ denote the $k$-algebra of binomial coefficients. Let us show that every finitely generated $A$-algebra that is a field is a finitely generated $A$-module. This is one of the characterisations of Jacobson rings, see e.g. \cite[Ch.5 Exercise 25]{MR0242802}. Let $K$ be such a field and let $\upsilon$ be the homomorphism of $k$-algebras $$A \overset{\upsilon}{\to} K, a \mapsto a.1.$$ 
Since ${s \choose p^e}^p={s \choose p^e}$ for all $e\geq0,$ its image by $\upsilon$ is in the prime field and accordingly $\upsilon(A)=k.$ It is well-known that fields are Jacobson rings. Therefore $K$ is a finitely generated $k$-module. Hence it is a finitely generated $A$-module. We have thus proved that $A$ is a Jacobson ring.

(\ref{item: radical}) Let $f$ be an element of $\mathbb{F}_p[{s \choose p^0}, {s \choose p^1}, {s \choose p^2}, \cdots] \cong C(\mathbb{Z}_p, \mathbb{F}_p).$ Clearly $f^{p^n}=f,$ for all natural numbers $n.$ Suppose that $f$ is in the radical of an ideal $I,$ i.e. $f^N$ is in $I$ for some $N\geq0.$ Let $n\geq0$ be such that $p^n\geq N.$ We have that $f=f^{p^n}$ is in $I.$ Thus the ideal $I$ is radical. Since by (\ref{item: Jacobson}), the algebra of binomial coefficients is radical, we have that every radical ideal is the intersection of maximal ideals. Hence every ideal of $\mathbb{F}_p[{s \choose p^0}, {s \choose p^1}, {s \choose p^2}, \cdots]$ is the intersection of the maximal ideals containing it. It is thus characterised by its set of roots.\qedhere

\end{proof}

\subsection{Definition of the $b$-function} 
Let $k$ be a field of characteristic $p>0$ and let $R$ be a commutative $k$-algebra. 
Recall that we denote by $D_R$ the ring of Grothendieck differential operators of $R$ over $k.$ 

\begin{defi} 
Let $R$ be a smooth commutative $k$-algebra and let $f$ be an element of $R.$ Denote by $R[t]$ the ring of polynomials in one variable $t$ over $R.$
\begin{enumerate}
\item
There is an inclusion of left $D_{R[t]}$-modules $R[t] \subset R[t][\frac{1}{f-t}].$ Let \emph{$B_f$} be the quotient left $D_R{[t]}$-module $R[t][\frac{1}{f-t}]/R[t].$

\item
Let $D_R[\nu_e;e \in \mathbb{N}]$ be the subring of $D_{R[t]}$ generated by $D_R$ and the higher Euler operators $\nu_e:=\frac{d}{dt}^{[p^e]}t^{p^e},$ for all $e\geq0.$
The \emph{delta function of $f$} is the class $\delta_f$ of $\frac{1}{f-t}$ in $B_f.$ Let \emph{$M_f$} be the left $D_R[\nu_e;e \in \mathbb{N}]$-submodule of $B_f$ generated by $\delta_f.$

\end{enumerate}
\end{defi}

\begin{lmm} \label{lmm:tEuler} For all natural numbers $e,$ we have the following identity in $D_{\mathbb{F}_p[t]}:$
$$[\nu_e,t]=t \Pi_{j=0}^{j=e-1} (1- \nu_j^{p-1}),$$
where the product over the empty set, i.e. for $e=0,$ is $1.$ In particular, the k-algebra of higher Euler operators $\Gamma_k$ satisfies $\Gamma_kt=t\Gamma_k$ in $D_{k[t]}.$\end{lmm}

\begin{proof} For the identity to hold in $D_{\mathbb{F}_p[t]},$ it is enough that both sides agree when evaluated at monomials $t^N,$ for all natural numbers $N\geq0.$ 

One easily sees that,
for the left-hand side: $$[\nu_e,t](t^N):=(\nu_et-t\nu_e)(t^N)= {N \choose p^e-1}t^{N+1},$$ and for the right-hand side: $$t \Pi_{j=0}^{j=e-1} (1- \nu_j^{p-1})(t^N)=\Pi_{j=0}^{j=e-1} (1- ({N\choose p^j}+1)^{p-1})t^{N+1}.$$ Since $n^{p-1}=1\text{ mod }p$ if and only if $n\not=0\text{ mod }p,$ one has that 
$$\Pi_{j=0}^{j=e-1} (1- ({N\choose p^j}+1)^{p-1})=0\text{ mod }p$$ if and only if there is a number $j$ in $\{0,\dots, e-1\}$ such that ${N \choose p^j}\not=p-1\text{ mod }p.$
By Lucas' Theorem, this is the case if and only if $$N\not=\Sigma_{j=0}^{j=e-1}(p-1)p^j\text{ mod }p^e.$$
Since $\Sigma_{j=0}^{j=e-1}(p-1)p^j=p^e-1,$ it is thus clear using Lucas' Theorem again that both sides of the identity evaluated at $t^N$ vanish if $N\not=p^e-1\text{ mod }p^e$ and that they are equal to $t^{N+1},$ otherwise.
This concludes the proof of the identity.

The inclusion $\Gamma_kt\subset t\Gamma_k$ follows immediately from the identity. The reverse inclusion $t\Gamma_k \subset \Gamma_kt$ follows from the identity and an easy induction. Namely, it suffices to show that for all natural numbers $m, n_0, \dots, n_m,$ the monomial $t\nu_m^{n_m}\dots\nu_0^{n_0}$ is in $\Gamma_kt.$ Let us well-order the finite sequences $(m, n_m, \dots, n_0)$ by the lexicographic order and induct on it. It is clear that for the smallest element $(0,0), t\nu_0^0=t$ is in $\Gamma_kt.$ Consider a sequence $(m, n_m, \dots, n_0).$ If $n_m=0,$ then $t\nu_m^{n_m}\dots\nu_0^{n_0}=t\nu_{m-1}^{n_{m-1}}\dots\nu_0^{n_0},$ corresponding to the smaller $(m-1,n_{m-1}, \dots, n_0).$ It is thus in $\Gamma_kt$ by the induction hypothesis. Suppose that $n_m\geq1,$ then by the identity, 

$\begin{array}{lll} 
t\nu_m^{n_m}\dots\nu_0^{n_0} &=& t\nu_m\nu_m^{n_m-1}\dots\nu_0^{n_0} \\&=& \nu_mt\nu_m^{n_m-1}\dots\nu_0^{n_0}- t \Pi_{j=0}^{j=m-1} (1- \nu_j^{p-1})\nu_m^{n_m-1}\dots\nu_0^{n_0},
\end{array}$

\noindent which is in $\Gamma_kt$ by the induction hypothesis. This concludes the proof of the lemma.\qedhere
 
\end{proof}

\begin{cor} \label{cor:tM_f} The left $D_R$-submodule $tM_f$ of $M_f$ is stable under the action of the higher Euler operators. Hence the quotient $N_f:= M_f/tM_f$ is a left $D_R[\nu_e;e \in \mathbb{N}]$-module.
\end{cor}

\begin{proof} By Lemma \ref{lmm:tEuler}, $\Gamma_ktM_f=t\Gamma_kM_f=tM_f.$\qedhere

\end{proof}

We can now give the definition of the $b$-function.

\begin{defi} \label{defi: N_f}
Let $k$ be a field of characteristic $p>0,R$ a smooth $k$-algebra and let $f$ be an element of $R.$ 
The left $D_R[\nu_e;e \in \mathbb{N}]$-module $N_f$ is in particular a $\Gamma_k$-module. Recall that we fixed in Lemma \ref{lmm: gamma} an isomorphism $$k[{s \choose p^0}, {s \choose p^1}, {s \choose p^2}, \cdots] \overset{\gamma}{\to}\Gamma_k.$$ 
We denote by $N_f^{\gamma}$ the $k[{s \choose p^0}, {s \choose p^1}, {s \choose p^2}, \cdots]$-module deduced from $N_f$ by the isomorphism $\gamma.$ 
The \emph{$b$-function} $b_f$ of $f$ is the annihilator of $N_f^{\gamma}$ in $k[{s \choose p^0}, {s \choose p^1}, {s \choose p^2}, \cdots].$
\end{defi}

\begin{rmk} 
We note that the definitions of $B_f, M_f,$ and hence of $N_f,$ are local on $Spec(R).$ This allows us to globalise the definition of the $b$-function to $f$ a function on a smooth $k$-variety, for example. We leave the details to the reader.
\end{rmk}

\subsection{Bounded level versions of $N_f$} \label{subsection: bounded level} 

Here we start analysing $N_f$ by considering differential operators of bounded level. Let us recall some  definitions. 
From here on, let us suppose that the base field $k$ is perfect.

\begin{defi} Let $p$ be a prime number. Let $R$ be a commutative algebra of characteristic $p.$ 
The \emph{Frobenius endomorphism} of $R$ is the ring endomorphism $F$ of $R$ raising elements to their $p$-th power, $$R \overset{F}{\to}R,$$ 
$$r \mapsto r^p.$$ 

\end{defi}

\begin{defi} Let $k$ be a perfect field of characteristic $p>0, R$ a smooth $k$-algebra and $n$ a natural number $\geq0.$ Let $R^{(n+1)}$ be the $R$-module $R\overset{F^{n+1}}{\to}R.$ The ring $D_R^{(n)}$ of \emph{differential operators of level $n$} on $Spec(R)$ is the ring of $R$-linear endomorphisms of $R^{(n+1)}.$
\end{defi}

\begin{prop} Let $l$ be a natural number. The differential operators of level $l$ are differential operators. Namely the canonical inclusion of $D_R^{(l)}$ in $End_k(R)$ factors through $D_R.$ Moreover, $D_R=\bigcup_{n\geq0}D_R^{(n)}.$
\end{prop}

\begin{proof} This is well-known. For example, it follows from \cite[Proposition 2.2.7]{MR1373933} and  \cite[(2.2.1.7)]{MR1373933}.\qedhere\end{proof}

We clearly have that the differential operators of level $n$ in $D_R[\nu_e;e \in \mathbb{N}]$ are $D_{R[t]}^{(n)} \cap D_R[\nu_e;e \in \mathbb{N}]= D_R^{(n)}[\nu_e; 0\leq e\leq n].$ It is thus natural to consider the subsequent bounded level versions of $M_f$ and $N_f.$

\begin{defi} Let $l$ be a natural number. The left $D_R^{(l)}[\nu_e; 0\leq e\leq l]$-module \emph{$M_f^{(l)}$} is $$M_f^{(l)}:= D_R^{(l)}[\nu_e; 0\leq e\leq l]\delta_f \subset M_f.$$
By Lemma \ref{lmm:tEuler}, we have that the left $D_R^{(l)}$-submodule $tM_f^{(l)}$ of $M_f^{(l)}$ is a left $D_R^{(l)}[\nu_e; 0\leq e\leq l]$-submodule. The left $D_R^{(l)}[\nu_e; 0\leq e\leq l]$-module \emph{$N_f^{(l)}$} is set to be the quotient, $$N_f^{(l)}:=M_f^{(l)}/tM_f^{(l)}.$$
\end{defi}

Let $l$ be a natural number and let $M$ be a $R$-module.  
By functoriality of the tensor product $-\otimes_R M,$ the \emph{Frobenius pull-back} $R$-module $({F^{l+1}})^*M:= R^{(l+1)} \otimes_R M$ is acted upon on the left by the ring of $R$-linear endomorphisms of $R^{(l+1)},$ i.e. by $D_R^{(l)}.$ The Frobenius pull-back is actually an equivalence of categories.

\begin{prop} Let $k$ be a perfect field of characteristic $p>0$ and $R$ a smooth $k$-algebra. Let $l$ be a natural number. The Frobenius pull-back $(F^{l+1})^*$ induces an equivalence from the category of finitely generated $R$-modules to the category of finitely generated left $D_R^{(l)}$-modules.\end{prop}

\begin{proof} This is a well-known instance of a Morita equivalence. For example, see \cite[Proposition 3.2.(1)]{MR883872}.\qedhere\end{proof}

Subsection \ref{subsection: N_f test} is devoted to describing a preimage of $N_f^{(l)}$ under this equivalence.

\subsection{Test ideals} In this subsection and the next, unless otherwise mentioned, we use the following notations:  $k$ is a perfect field of characteristic $p>0, R$ is a smooth $k$-algebra and $f$ is an element of $R.$

Let us recall the definition of the test ideals of $f$ and sum up the properties we will use.

\begin{defi} Let $e$ be a natural number and $J$ an ideal of $R.$ 
The \emph{$p^e$-th root ideal} $J^{[\frac{1}{p^e}]}$ of $J$ is the intersection of all the ideals $I$ of $R$ such that $I^{[p^e]}\supset J,$ where $I^{[p^e]}$ is the ideal of $R$ generated by the $p^e$-th powers of the elements of $I.$ 

\end{defi}

\begin{lmm} 
Let $\lambda$ be a non-negative real number and $e$ a natural number. There is an inclusion of ideals of $R,$

$$(f^{\lceil\lambda p^e\rceil}R)^{[\frac{1}{p^e}]} \subset (f^{\lceil\lambda p^{e+1}\rceil}R)^{[\frac{1}{p^{e+1}}]},$$ where $\lceil \rceil$ is the ceiling function which maps a real number to the smallest greater integer.

\end{lmm}

\begin{proof} The lemma is a special case of \cite[Lemma 2.8]{MR2492440}.\qedhere\end{proof}

\begin{defi} Let $\lambda$ be a non-negative real number. 
The \emph{test ideal of exponent $\lambda$ of $f$} is the following ideal of $R,$ $$\tau(f^\lambda):= \bigcup_{e\geq0} (f^{\lceil\lambda p^e\rceil}R)^{[\frac{1}{p^e}]}.$$

\end{defi}

The test ideal of exponent $\lambda$ of $f$ clearly decreases as $\lambda$ increases.

\begin{lmm} Let $\lambda$ be a non-negative real number. There is a real number $\epsilon>0$ such that the test ideal of exponent $\lambda'$ of $f$ is $\tau(f^{\lambda}),$ for every non-negative real number $\lambda'$ in the interval $[\lambda,\lambda+\epsilon).$ \end{lmm}

\begin{proof} This is a special case of \cite[Corollary 2.16]{MR2492440}.\qedhere\end{proof}

\begin{defi} A \emph{$F$-jumping exponent of $f$} is a positive real number $\lambda$ such that for all real numbers $\epsilon >0$ such that the difference $\lambda-\epsilon$ is positive, $\tau(f^{\lambda - \epsilon}) \neq \tau(f^{\lambda}).$
\end{defi}

The following finiteness result is crucial to us.

\begin{thm} \label{thm:F-jumps} The set of $F$-jumping exponents of $f$ is 

\begin{enumerate}
\item \label{thm:F-jumps discrete} a discrete subset of $\mathbb{R}$
\item \label{thm:F-jumps rational} a subset of $\mathbb{Q}.$
\end{enumerate}
\end{thm}

\begin{proof} This is \cite[Theorem 3.1]{MR2492440}, for a principal ideal.\qedhere \end{proof}

We will also use the following definition.

\begin{defi} Let $\lambda$ be a positive real number. The \emph{test ideal immediately preceding $\tau(f^{\lambda})$} is $\tau(f^{\lambda^{-}}):= \cap_{\mu<\lambda}\tau(f^{\mu}).$ 
The \emph{associated graded to the test ideal filtration of $R$ with respect to $f$} is 

$$\mathfrak{gr}_{\lambda \in (0,1]}(\tau(f^{\lambda})):= \bigoplus_{\lambda \in (0,1]}\tau(f^{\lambda^{-}})/\tau(f^{\lambda}).$$ 
\end{defi}

\begin{lmm} \label{lmm:associated graded test} Let $\lambda$ be a positive real number. Denote by $\lambda'$ the supremum of the subset $S_{\lambda}$ of the real numbers containing $0$ and the $F$-jumping exponents of $f$ strictly smaller than $\lambda.$
Since $S_\lambda$ is discrete by (\ref{thm:F-jumps discrete}) of Theorem \ref{thm:F-jumps}, $\lambda'$ is in $S_\lambda.$

The test ideal immediately preceding $\tau(f^\lambda)$ is $\tau(f^{\lambda})$ if and only if $\lambda$ is not a $F$-jumping exponent of $f.$ If $\lambda$ is a $F$-jumping exponent of $f,$ then the test ideal immediately preceding $\tau(f^\lambda)$ is $\tau(f^{\lambda'}).$ \end{lmm}

\begin{proof} If $\lambda$ is not a $F$-jumping exponent, then there exists a real number $\epsilon>0$ such that $\tau(f^{\lambda-\epsilon})=\tau(f^{\lambda}).$ Since the filtration by test ideals is decreasing, $\tau(f^{\lambda^-}):= \cap_{\mu<\lambda}\tau(f^{\mu})=\tau(f^\lambda).$ 

If $\lambda$ is a $F$-jumping exponent, then for all real numbers $r$ in $[\lambda',\lambda), \tau(f^r)=\tau(f^{\lambda'}).$ Thus $\tau(f^{\lambda^-})=\tau(f^{\lambda'}).$\qedhere
\end{proof}

\begin{cor} \label{cor:associated graded test}
Let $\{\lambda_1< \dots< \lambda_n\}$ be the $F$-jumping exponents of $f$ in the interval $(0,1].$ For all integers $i$ in $\{1, \dots, n\},$ the test ideal immediately preceding $\tau(f^{\lambda_i})$ is $\tau(f^{\lambda_{i-1}}),$ where we have set $\lambda_0=0.$ Thus the non-zero summands of the associated graded to the test ideal filtration of $R$ with respect to $f$ are indexed by the F-jumping exponents of f in the interval $(0,1]$ and  
$$\mathfrak{gr}_{\lambda \in (0,1]}(\tau(f^{\lambda}))= \bigoplus_{i=1}^{i=n}\tau(f^{\lambda_{i-1}})/\tau(f^{\lambda_i}).$$ \end{cor}

\begin{proof} It follows directly from Lemma \ref{lmm:associated graded test}\qedhere\end{proof}

\subsection{$N_f^{(l)}$ and test ideals} \label{subsection: N_f test}

By explicit computations, Musta{\c{t}}{\u{a}} observed the following.

\begin{prop} \label{prop:Df} 
Let $k$ be a perfect field of characteristic $p>0,R$ a smooth $k$-algebra and let $f$ be an element of $R.$ 
Let $l$ be a natural number. 

\begin{enumerate}
\item \label{item: M_f}
There is a natural isomorphism of left $D_R^{(l)}$-modules,
$$M_f^{(l)}\cong \bigoplus_{0\leq n <p^{l+1}}D_R^{(l)}f^n,$$ where for all elements $r$ of $R, D_R^{(l)}r$ is the left $D_R^{(l)}$-submodule of $R$ generated by $r.$
It induces an isomorphism
$N_f^{(l)}\cong \bigoplus_{0\leq n <p^{l+1}}D_R^{(l)}f^n/D_R^{(l)}f^{n+1}.$

\item \label{item: Euler}
Let $n'$ be a natural number $<p^{l+1}$ and let $n'=\sum_{0\leq e\leq l}a_ep^e$ be its base $p$ expansion. The action of the higher Euler operators on $N_f^{(l)}$ transported by the isomorphism of (\ref{item: M_f}) to an action on $\bigoplus_{0\leq n <p^{l+1}}D_R^{(l)}f^n/D_R^{(l)}f^{n+1}$ stabilises $D_R^{(l)}f^{n'}/D_R^{(l)}f^{n'+1}$ 
and for all $e$ in $\{0,\dots, l\},$ $\nu_e$ acts on $D_R^{(l)}f^{n'}/D_R^{(l)}f^{n'+1}$ by $-a_e.$

\item \label{item: maps}
The natural inclusion $M_f^{(l)} \subset M_f^{(l+1)}$ is transported by the isomorphisms of level $l$ and $l+1$ of (\ref{item: M_f}) to $$\bigoplus_{0\leq n <p^{l+1}}D_R^{(l)}f^n \to \bigoplus_{0\leq m <p^{l+2}}D_R^{(l+1)}f^m$$
$$\sum_{0\leq n <p^{l+1}}g_n \mapsto \sum_{0\leq n <p^{l+1}}\sum_{0\leq j<p}(-1)^j {{p-1} \choose j} (g_nf^{jp^{l+1}})_{n+jp^{l+1}},$$ where the subscript of an element indicates the direct summand to which it belongs.

\end{enumerate}
\end{prop}

\begin{proof} (\ref{item: M_f}) The first isomorphism is shown in \cite[Proposition 6.1]{MR2469353} and the second in \cite[Corollary 6.5]{MR2469353}. 

(\ref{item: Euler}) See \cite[Corollary 6.5]{MR2469353}. 

(\ref{item: maps}) It is \cite[Remark 5.7]{MR2469353}.\qedhere\end{proof}

Recall that we want to express a preimage of $N_f^{(l)}$ under the Frobenius pull-back $(F^{l+1})^*.$ We start with $D_R^{(l)}f^n.$

\begin{lmm} \label{lmm:Df test} There is a canonical isomorphism of left $D_R^{(l)}$-modules coming from an elementary equality of ideals, $$D_R^{(l)}f^n\cong (F^{l+1})^*\tau(f^{\frac{n}{p^{l+1}}}).$$\end{lmm}

\begin{proof} 
By \cite[Lemma 2.2]{MR2538604}, the left $D_R^{(l)}$-submodule of $R$ generated by $f^n$ is $$((f^n)^{[\frac{1}{p^{l+1}}]})^{[p^{l+1}]}.$$
The latter is none other than $\tau(f^{\frac{n}{p^{l+1}}})^{[p^{l+1}]}\cong (F^{l+1})^*\tau(f^{\frac{n}{p^{l+1}}}),$ by \cite[Lemma 2.1]{MR2538604}.\qedhere\end{proof}

\begin{defi} Let $l$ be a natural number. We say that $l$ \emph{separates the $F$-jumping exponents of $f$} if the partition of $(0,1]$ in intervals of length $\frac{1}{p^{l+1}}$ separates the $F$-jumping exponents of $f,$ i.e. for each natural number $n< p^{l+1},$ the interval $(\frac{n}{p^{l+1}},\frac{n+1}{p^{l+1}}]$ contains at most one $F$-jumping exponent of $f.$\end{defi}

\begin{prop} \label{prop:Frob pull-back} Let $l$ be a natural number. Recall that by (\ref{thm:F-jumps discrete}) of Theorem \ref{thm:F-jumps}, there are only finitely many F-jumping exponents of $f$ in the unit interval $[0,1]$ and suppose that $l$ is large enough to separate the $F$-jumping exponents of $f.$ Then the canonical isomorphism of Lemma \ref{lmm:Df test} induces an isomorphism of left $D_X^{(l)}$-modules, $$N_f^{(l)}\cong (F^{l+1})^*(\mathfrak{gr}_{\lambda \in (0,1]}(\tau(f^{\lambda}))).$$
Hence for every natural number $l'$ greater than $l,$ the map induced from that in (\ref{item: maps}) of Proposition \ref{prop:Df} transports to a morphism $$(F^{l'+1})^*(\mathfrak{gr}_{\lambda \in (0,1]}(\tau(f^{\lambda}))) \overset{a_{l'+1}}{\to} (F^{l'+2})^*(\mathfrak{gr}_{\lambda \in (0,1]}(\tau(f^{\lambda})))$$ and there is an isomorphism of left $D_R$-modules
$$N_f=\varinjlim_{l'\geq l}N_f^{(l')}\cong \varinjlim_{l'\geq l}(F^{l'+1})^*(\mathfrak{gr}_{\lambda \in (0,1]}(\tau(f^{\lambda}))),$$ where $\varinjlim_{l'\geq l}(F^{l'+1})^*(\mathfrak{gr}_{\lambda \in (0,1]}(\tau(f^{\lambda})))$ denotes the direct limit with respect to these morphisms. 
\end{prop}

\begin{proof} By (\ref{item: M_f}) of Proposition \ref{prop:Df} and Lemma \ref{lmm:Df test}, for all natural numbers $l,$ we have that $$N_f^{(l)}\cong \bigoplus_{0\leq n <p^{l+1}}(F^{l+1})^*\tau(f^{\frac{n}{p^{l+1}}})/(F^{l+1})^*\tau(f^{\frac{n+1}{p^{l+1}}})$$ and thus that 
$$N_f^{(l)}\cong (F^{l+1})^*(\bigoplus_{0\leq n <p^{l+1}}\tau(f^{\frac{n}{p^{l+1}}})/\tau(f^{\frac{n+1}{p^{l+1}}})),$$ by flatness of the Frobenius pull-back.

Let $n$ be a natural number $<p^{l+1}.$ If there is no $F$-jumping exponent of $f$ in the interval $(\frac{n}{p^{l+1}},\frac{n+1}{p^{l+1}}],$ then $\tau(f^{\frac{n}{p^{l+1}}})=\tau(f^{\frac{n+1}{p^{l+1}}}).$ Thus the quotient $\tau(f^{\frac{n}{p^{l+1}}})/\tau(f^{\frac{n+1}{p^{l+1}}})$ vanishes.
On the other hand, if there is an $F$-jumping exponent of $f$ in the interval $(\frac{n}{p^{l+1}},\frac{n+1}{p^{l+1}}],$ then the quotient $\tau(f^{\frac{n}{p^{l+1}}})/\tau(f^{\frac{n+1}{p^{l+1}}})$ is equal to $\tau(f^{\lambda^-})/\tau(f^\lambda),$ where $\lambda$ is the $F$-jumping coefficient in $(\frac{n}{p^{l+1}},\frac{n+1}{p^{l+1}}],$ unique by the hypothesis on $l.$

Finally, since, for varying $n,$ the intervals $(\frac{n}{p^{l+1}},\frac{n+1}{p^{l+1}}]$ form a partition of $(0,1],$ all $F$-jumping coefficients of $f$ in $(0,1]$ appear in this way. Thus one has that $$\bigoplus_{0\leq n <p^{l+1}}\tau(f^{\frac{n}{p^{l+1}}})/\tau(f^{\frac{n+1}{p^{l+1}}}))= \bigoplus_\lambda \tau(f^{\lambda^-})/\tau(f^\lambda),$$ where the direct sum on the right-hand side is over the $F$-jumping exponents of $f$ in the interval $(0,1].$ 
One applies Corollary \ref{cor:associated graded test} to conclude the proof of the proposition.\qedhere 
\end{proof}

In the next section,  we use the relation between $N_f$ and the $F$-jumping exponents of $f$ to get information about the roots of its $b$-function.

\section{Relation to test ideals and rationality of the roots} \label{section:test ideals}

\subsection{Preliminaries on $p$-adic and $\frac{1}{p}$-adic expansions}

We will express the $p$-adic expansions of the roots of the $b$-function of $f$ in terms of the ``$\frac{1}{p}$-adic expansions" of its $F$-jumping exponents. Let us precise what we mean. 
  
\begin{defi} Let $p$ be a prime number and let $r$ be a positive real number. The \emph{$\frac{1}{p}$-adic expansion} of the number $r$ is its unique base $p$ expansion which does not have infinitely many consecutive zero coefficients. It is written $r=\sum_{-b\leq n} r_n(\frac{1}{p})^n,$ where the coefficients $r_n$ belong to the set $\{0, \dots, p-1\}$ and $b$ is a natural number. The \emph{$\frac{1}{p}$-adic fractional part} of $r$ is $frac_{\frac{1}{p}}(r):= \sum_{n\geq 1} r_n(\frac{1}{p})^n$ and its \emph{$\frac{1}{p}$-adic integer part} is $[r]_{\frac{1}{p}}:= \sum_{n=-b}^{n=0} r_n(\frac{1}{p})^n.$
\end{defi}

\begin{defi} A positive real number $r$ of $\frac{1}{p}$-adic expansion $r=\sum_{-b\leq n} r_n(\frac{1}{p})^n$ is \emph{$\frac{1}{p}$-adically periodic} if there is a natural number $n'\geq1$ such that the sequence $(r_n)_{n\geq n'}$ is repeating. Let $n''$ be the smallest such natural number. The \emph{$\frac{1}{p}$-adic anteperiod} of $r$ is $n''-1.$ The \emph{$\frac{1}{p}$-adic period} of $r$ is the period $\rho_r$ of the sequence $(r_n)_{n\geq n''}$ and the word $r_{n''}\dots r_{n''+\rho_r-1}$ is its $\frac{1}{p}$-adic repetend.
A $\frac{1}{p}$-adically periodic number is said to be \emph{purely $\frac{1}{p}$-adically periodic} if its $\frac{1}{p}$-adic anteperiod is $0.$

\end{defi}

The following lemma characterises the periodicity properties of $\frac{1}{p}$-adic expansions.
 
\begin{lmm} \label{lmm:periodic expansion} Let $p$ be a prime number.

\begin{enumerate}
\item \label{item:periodic} A positive real number is $\frac{1}{p}$-adically periodic if and only if it is rational.
\item \label{item:purely periodic} A positive real number is purely $\frac{1}{p}$-adically periodic if and only if it is in $\mathbb{Z}_{(p)}.$
\end{enumerate}
\end{lmm}

\begin{proof} (\ref{item:periodic}) To prove that $\frac{1}{p}$-adic periodicity implies rationality, one notices that for all $e\geq1, \sum_{n>0}\frac{1}{p^{en}}= \frac{1}{p^e-1}.$ The rest of the proof is very standard and left to the reader.

(\ref{item:purely periodic}) The assertion is invariant under shifts by natural numbers. Hence it suffices to prove it for positive real numbers in the interval $(0,1].$ 

If a number $r \in (0, 1]$ is purely $\frac{1}{p}$-adically periodic, then its $\frac{1}{p}$-adic expansion is of the form $r_0\sum_{n>0}\frac{1}{p^{en}}= \frac{r_0}{p^e-1},$ where $e$ is a natural number $>1$ and $r_0$ is smaller than $p^e.$ Thus $r$ is in $\mathbb{Z}_{(p)}.$

Suppose that $r$ is in $\mathbb{Z}_{(p)}\cap (0, 1]$ and let $r=\frac{a}{b}$ be its reduced rational expression. Then it is well-known that, since $b$ is prime to $p,$ there is a positive natural number $e$ such that $b$ divides $p^e-1.$ Thus $r= \frac{a'}{p^e-1}=a'\sum_{n>0}\frac{1}{p^{en}},$ with $0<a'<p^e.$ Hence $r$ is purely $\frac{1}{p}$-adically periodic.\qedhere\end{proof}

To each $\frac{1}{p}$-adically periodic number  
correspond finitely many ``conjugated" numbers. 

\begin{defi} The \emph{$\frac{1}{p}$-adic conjugates} of  a $\frac{1}{p}$-adically periodic number $r$ are the numbers whose $\frac{1}{p}$-adic expansion is obtained from that of $r$ by a cyclic permutation of the letters of its $\frac{1}{p}$-adic repetend.

\end{defi}

Thus for $l$ the anteperiod of $r$ and $d$ its period, the $\frac{1}{p}$-adic conjugates of $r= \frac{b_1}{p}+\dots+\frac{b_l}{p^l}+\frac{a_1}{p^{l+1}}+ \frac{a_2}{p^{l+2}}+ \dots+ \frac{a_{d-1}}{p^{l+d-1}}+ \frac{a_d}{p^{l+d}}+\frac{a_1}{p^{l+d+1}}+\dots$ are 

$$\{\frac{b_1}{p}+\dots+\frac{b_l}{p^l}+\frac{a_1}{p^{l+1}}+ \frac{a_2}{p^{l+2}}+ \dots+ \frac{a_{d-1}}{p^{l+d-1}}+ \frac{a_d}{p^{l+d}}+\frac{a_1}{p^{l+d+1}}+\dots,$$ 
$$\frac{b_1}{p}+\dots+\frac{b_l}{p^l}+\frac{a_d}{p^{l+1}}+ \frac{a_1}{p^{l+2}}+ \dots+ \frac{a_{d-2}}{p^{l+d-2}}+ \frac{a_{d-1}}{p^{l+d}}+\frac{a_d}{p^{l+d+1}}+\dots, \dots, $$$$
\frac{b_1}{p}+\dots+\frac{b_l}{p^l}+\frac{a_2}{p^{l+1}}+ \frac{a_3}{p^{l+2}}+ \dots+ \frac{a_{d}}{p^{l+d-1}}+ \frac{a_1}{p^{l+d}}+\frac{a_2}{p^{l+d+1}}+\dots\}.$$

We now present a pure periodicity lemma which will be useful in our description of the roots of the $b$-function. We start with a definition.

\begin{defi} Let $n$ be a positive natural number and let $a_1, \dots, a_n$ 
be elements of $\{0, \dots, p-1\}.$  
Set $a$ to be the rational number $a:=\sum_{i=1}^{i=n}\frac{a_i}{p^i}.$ A positive real number $\lambda$ \emph{begins with $a$} if its $\frac{1}{p}$-adic fractional part begins with $a,$ i.e. if there are coefficients $\lambda_j$ for all j greater than n such that 
$$frac_{\frac{1}{p}}(\lambda)= \sum_{i=1}^{i=n}\frac{a_i}{p^i} + \sum_{j>n}\frac{\lambda_j}{p^j}.$$  
\end{defi}

\begin{lmm} \label{lmm:infinity}

Let $s$ be a sequence $s: \mathbb{N}_0 \to \{0, \dots, p-1\}, n\mapsto s(n)$
and let $\Lambda$  
be a finite set of rational numbers. For all positive natural numbers $l,$ set $s_l:= \sum_{i=1}^{i=l}\frac{s(l-i+1)}{p^i} .$ 

Suppose that for all natural numbers $N$ large enough, there is an element $\lambda$ of $\Lambda$ beginning with $s_N.$ Then there is a number $L$ such that for all $n$ greater than $L,$ there is a purely $\frac{1}{p}$-adically periodic element of $\Lambda$ beginning with $s_n.$ 
\end{lmm}

\begin{proof} Since $\Lambda$ is a finite set, there is a number $L$ greater than $N$ such that for each $l$ greater than $L,$ there is an element $\lambda$ in $\Lambda$ beginning with $s_l$ and beginning with $s_{l'}$ for infinitely many $l'$ greater than $l.$ 
Let us prove that such a rational number $\lambda$ has to be purely $\frac{1}{p}$-adically periodic. 

Clearly, one may suppose that $\lambda$ coincides with its $\frac{1}{p}$-adic fractional part. Let $\lambda=\sum_{j\geq1}\frac{\lambda_j}{p^j}$ be its $\frac{1}{p}$-adic expansion.
By (\ref{item:periodic}) of Lemma \ref{lmm:periodic expansion}, $\lambda$ is $\frac{1}{p}$-adically periodic. Let $n$ be its anteperiod and $\rho>0$ be its period. We claim that for all numbers $j\geq1, \lambda_{j+\rho}=\lambda_j.$ This indeed implies that $\lambda$ is purely $\frac{1}{p}$-adically periodic. We will show that the hypothesis on the beginning of $\lambda$ allows one to embed the beginning of its $\frac{1}{p}$-adic expansion into its repeating part, thus forcing pure periodicity.

More precisely, since $\lambda_{j+\rho}=\lambda_j$ for all $j>n$ by definition of the anteperiod $n,$ it is enough to show that $\lambda_{j+\rho}=\lambda_j$ for all $1\leq j\leq n.$ Consider $l$ greater than $max\{n+\rho,L\}$ such that $\lambda$ begins by $s_l.$ Thus for all $1\leq i\leq n+\rho, \lambda_i$ is equal to $s(l-i+1).$ Since there is also $l'$ greater than $l+n$ such that $\lambda$ begins with $s_{l'},$ we have that  for all $1\leq i\leq n,$ $$\lambda_i=s(l-i+1)= s(l'-(l'-l+i)+1)=\lambda_{l'-l+i}.$$ Moreover $l'-l$ is not smaller than the anteperiod $n,$ thus $l'-l+i>n.$ Hence $\lambda_{l'-l+i}$ is equal to $\lambda_{l'-l+i+\rho}.$ Using again that $\lambda$ begins with $s_{l'},$ we get that the latter $\lambda_{l'-l+i+\rho}=s(l'-(l'-l+i+\rho)+1)= s(l-(i+\rho)+1).$ Recall that $\lambda$ begins also with $s_l$ and thus that $s(l-(i+\rho)+1)=\lambda_{i+\rho}.$ In conclusion, we have proved that for each $i$ between $1$ and $n,$ $$\lambda_i=\lambda_{l'-l+i}=\lambda_{l'-l+i+\rho}=\lambda_{i+\rho}.$$ This completes the proof of the lemma.\qedhere
\end{proof}

Finally, let us express the $p$-adic expansion of a number in $\mathbb{Z}_{(p)}$ in terms of the $\frac{1}{p}$-adic expansion of its opposite. Note that a positive real number is in the interval $(0,1]$ if and only if it equals its $\frac{1}{p}$-adic fractional part. We have the following.

\begin{lmm} \label{lmm:expansion of opposite} Let $r$ be a positive real number in the interval $(0,1].$ Suppose that $r$ is purely $\frac{1}{p}$-adically periodic and let $r=\sum_{n\geq 1} \frac{r_n}{p^n}$ be its $\frac{1}{p}$-adic expansion.  
Let $d$ be its $\frac{1}{p}$-adic period. Then the $p$-adic expansion of $-r$ is $$-r=\sum_{n\geq 0} r_{d-n_d}p^{n},$$ where $n_d$ is the representative of $n\text{ mod }d$ in $\{0,\dots, d-1\}.$ 

\end{lmm}

\begin{proof} Let $d$ be the $\frac{1}{p}$-adic period of $r$ and let $r':= \sum_{i=1}^{i=d} r_i p^{d -i}.$ By pure periodicity, we have that $r= r'\sum_{e\geq 1} \frac{1}{p^{d e}}.$ Moreover $\sum_{e\geq 1} \frac{1}{p^{d e}}=\frac{1}{p^d-1}$ and the $p$-adic expansion of $-\frac{1}{p^d-1}$ is $\sum_{j\geq0}p^{dj}.$ Thus the $p$-adic expansion of $-r$ is equal to $$-r= r'\sum_{j\geq0}p^{dj}= \sum_{1\leq i\leq d; j\geq0} r_i p^{dj+d-i}= \sum_{n\geq 0} r_{d-n_d}p^{n}.$$\qedhere 

\end{proof}

\subsection{Unit $F$-modules} In the next subsection, we will show that $N_f$ seen as a left $D_R$-module is of a very particular type, namely it is a unit $F$-module. Let us first recall their definition and basic properties. 

From now on, unless otherwise mentioned, $k$ is a perfect field of characteristic $p>0, R$ is a smooth $k$-algebra, $f$ is an element of $R$ and $F$ is the Frobenius endomorphism of $R.$ 

\begin{defi}
Let $M$ be an $R$-module and $e$ a natural number. Let $R \overset{\mu_M}{\to} End_k(M)$ be the structure map. The \emph{$e$-th Frobenius direct image of $M$} is the $R$-module $F^e_*M$ whose underlying $k$-vector space is $M,$ with structure map $\mu_M\circ F^e: R \overset{F^e}{\to} R \overset{\mu_M}{\to} End_k(M).$
\end{defi}

\begin{defi} Let $M$ be an $R$-module and $e$ be a positive integer. An \emph{$F^e$-module} 
structure on $M$ is an $R$-linear morphism $F^e_M: M \to F^e_*M.$ If we do not want to specify $e,$ we call an $F^e$-module a \emph{Frobenius module} or \emph{$F$-module}. 
\end{defi}

\begin{defi} Let $e$ be a positive integer. The ring $R[F^e]$ is the ring generated by $R$ and an element $F^e,$ with the relations $F^er=r^{p^e}F^e,$ for all elements $r$ of $R.$ The data of an $F^e$-module structure on an $R$-module is clearly equivalent to that of a left $R[F^e]$-module structure, compatible with the action of $R.$ We say that an $F^e$-module is \emph{finitely generated} if it is so as a left $R[F^e]$-module. 
\end{defi}

\begin{defi} Let $e$ be a positive integer. An $F^e$-module $M$ is \emph{unit} if the adjoint map to the structure map $F^e_M,$ $$(F^e)^*M\overset{{F^e_M}'}{\to} M$$ is an isomorphism.  
\end{defi}

The following is well-known.

\begin{prop} Let $e$ be a positive integer. A unit $F^e$-module is canonically endowed with a structure of left $D_R$-module. 
\end{prop}

\begin{proof} Denote the unit $F^e$-module by $M.$ For all natural numbers $n,$ we have that $\gamma_n:= (F^{ne})^*({F^e_M}') \circ \dots \circ (F^e)^*({F^e_M}') \circ {F^e_M}'$ is an $R$-linear isomorphism $(F^{(n+1)e})^*M \overset{\gamma_n}{\to} M.$ As noted in Subsection \ref{subsection: bounded level}, $(F^{(n+1)e})^*M$ is endowed with a canonical left $D_R^{((n+1)e-1)}$-module structure. Hence so is $M,$ via the isomorphism $\gamma_n.$ Moreover, it is straightforward to check that those actions are compatible with the inclusions $D_R^{((n+1)e-1)} \subset D_R^{((n+2)e-1)}.$ Since $D_R= \bigcup_{n\geq0} D_R^{((n+1)e-1)}, M$ is thus canonically a left $D_R$-module.\qedhere
\end{proof}

There is a convenient way to generate unit $F$-modules, due to Lyubeznik (\cite{MR1476089}). 
\begin{prop} Let $e$ be a positive integer. Let $G$ be an $R$-module and $\beta$ be an $R$-linear morphism $G\overset{\beta}{\to} (F^e)^*G.$ Consider the direct limit of the direct system
$$G \overset{\beta}{\to} (F^e)^*G \overset{(F^e)^*\beta}{\to} (F^{2e})^*G  \overset{(F^{2e})^*\beta}{\to} \dots$$ 
and denote it $\varinjlim \beta.$ 
Then the $R$-module $\varinjlim \beta$ is naturally isomorphic to its pull-back $(F^e)^*\varinjlim \beta.$ Hence it is canonically endowed with a structure of unit $F^e$-module.
\end{prop}

\begin{proof} This is a direct consequence of the commutation of pull-back with direct limits.\qedhere
\end{proof}

\begin{defi} Let $e$ be a positive integer and $M$ a unit $F^e$-module. The data of a finitely generated $R$-module $G,$ an $R$-linear morphism $G\overset{\beta}{\to} (F^e)^*G$ and an isomorphism $\iota$ of Frobenius modules $\varinjlim \beta \to M$ is called a \emph{generator} of $M.$ We will often omit to mention the isomorphism $\iota.$
\end{defi}

In fact, all finitely generated unit $F$-modules appear this way, as noted by Emerton-Kisin. 

\begin{thm} Every finitely generated unit $F^e$-module has a generator.
\end{thm}

\begin{proof} This is a special case of \cite[Theorem 6.1.3]{MR2071510}.\qedhere\end{proof}

Here is a fundamental finiteness property of unit $F$-modules by which they stand out among $D_R$-modules. It is due to Lyubeznik.

\begin{thm} A finitely generated unit $F$-module is of finite length as a left $D_R$-module.
\end{thm}

\begin{proof} It follows immediately from \cite[Theorem 3.2]{MR1476089} and \cite[Theorem 5.6]{MR1476089}.\qedhere\end{proof}

\subsection{A unit $F$-structure on $N_f$} Here we continue the study of $N_f$ via its bounded level versions started in \ref{subsection: N_f test}.  
In particular, we show that the left $D_R$-module $N_f$ is a unit $F$-module and the higher Euler operators are compatible with its Frobenius endomorphism. 

\begin{prop} \label{prop:F-jumpDf}  Let $l$ be a natural number and $\lambda$ a $F$-jumping exponent of $f$ in $(0,1].$ Suppose that $l$ is large enough to separate the $F$-jumping exponents of $f.$ Then the composition of the isomorphisms of Propositions \ref{prop:Df}.\ref{item: M_f} and \ref{prop:Frob pull-back}, $$\bigoplus_{0\leq n <p^{l+1}}D_R^{(l)}f^n/D_R^{(l)}f^{n+1} \cong N_f^{(l)}\cong (F^{l+1})^*(\mathfrak{gr}_{\lambda \in (0,1]}(\tau(f^{\lambda})))$$ induces an isomorphism  $$D_R^{(l)}f^m/D_R^{(l)}f^{m+1} \cong (F^{l+1})^*(\tau(f^{\lambda^{-}})/\tau(f^{\lambda})),$$ for the unique $m$ such that $\frac{m}{p^{l+1}}$ is the truncated $\frac{1}{p}$-adic expansion of $\lambda.$

\end{prop}

\begin{proof} It is clear that $\lambda \in (\frac{n}{p^{l+1}},\frac{n+1}{p^{l+1}}]$ if and only if the $\frac{1}{p}$-adic expansion of $\lambda$ begins with $\frac{n}{p^{l+1}}.$ The proposition follows then easily from the definition of the isomorphism of Proposition \ref{prop:Frob pull-back}.\qedhere\end{proof}

Let us now study the structure maps of the inductive system for $N_f$ from Proposition \ref{prop:Frob pull-back} in more details. 
\begin{prop} \label{prop:structure} 
There are natural numbers $L$ and $N$ such that for all $F$-jumping exponents $\lambda$ of $f$ in $(0, 1]$ which are not in $\mathbb{Z}_{(p)},$ and for all $l>L$ and $e>N,$
the restriction to $(F^{l+1})^*(\tau(f^{\lambda^{-}})/\tau(f^{\lambda}))$ of the $e$-th iterate of the structure maps $$(F^{l+1})^*(\mathfrak{gr}_{\mu \in (0,1]}(\tau(f^{\mu}))) \to (F^{l+e+1})^*(\mathfrak{gr}_{\mu \in (0,1]}(\tau(f^{\mu})))$$ vanishes. 

\end{prop}

\begin{proof} We first take $L$ large enough so that $l$ separates the $F$-jumping exponents of $f.$ By Proposition \ref{prop:F-jumpDf} and the description of the structure maps in (\ref{item: maps}) of Proposition \ref{prop:Df}, for all positive $i,$ the $i$-th iterate of the structure maps restricts to $$(F^{l+1})^*(\tau(f^{\lambda^{-}})/\tau(f^{\lambda})) \to \bigoplus_{0\leq n< p^{l+i+1} \text{and } n = m \text{ mod }p^{l+1}} D_R^{(l+i)}f^n/D_R^{(l+i)}f^{n+1},$$ for the unique $m$ such that $\frac{m}{p^{l+1}}$ is the truncated $\frac{1}{p}$-adic expansion of $\lambda.$

We claim that there are numbers $N$  and $L$  
such that if $\lambda\not\in \mathbb{Z}_{(p)},$ then for all $i>N,$ and $l>L,$  the above map vanishes.  

Let us argue by contradiction and suppose that there are $l$ and $i,$ large at will, such that the above map does not vanish. In particular the target of the map is not trivial. Hence there is a number $n_i,$ congruent to $m$ modulo $p^{l+1},$ such that $\frac{n_i}{p^{l+i+1}}$ is the truncated $\frac{1}{p}$-adic expansion of an $F$-jumping exponent of $f.$ Thus $n_i= m + b_ip^{l+1},$ for a certain natural number $b_i=\Sigma_{j=0}^{j=i-1} (b_i)_{j+1}p^j$ and $$\frac{n_i}{p^{l+i+1}}= \frac{(b_i)_{i}}{p} +  \frac{(b_i)_{i-1}}{p^2}+ \dots+ \frac{(b_i)_1}{p^i}+\frac{m_1}{p^{i+1}}+\frac{m_2}{p^{i+2}}+\dots+\frac{m_{l+1}}{p^{i+l+1}},$$ where $\frac{m_1}{p}+\frac{m_2}{p^2}+\dots+\frac{m_{l+1}}{p^{l+1}}$ is the truncated $\frac{1}{p}$-adic expansion of $\lambda.$  
Since the map is an $i$-th iterate, for it not to vanish it is necessary that the corresponding $j$-th iterate does not vanish either, for all $j\leq i.$ Hence one may assume that as $i$ varies, the numbers $b_i$ are compatible with each others. Namely for all $i\leq i'$ such that $b_i$ and $b_{i'}$ are defined, one has that $b_{i'}$ is congruent to $b_i$ modulo $p^i.$ For all $1\leq j\leq i,$ let us denote $(b)_j$ the coefficient $(b_i)_j=(b_{i'})_j.$ 

For every $l$ as above, we define the sequence  
$s
:= (m_{l+1}, m_l, \dots, m_1, (b)_1, (b)_2, \dots).$ Let $\Lambda$ be the set of $F$-jumping exponents of $f$ in $(0,1].$ It is clear that for all $n$ large enough, there is an element of $\Lambda$ beginning with $s_n,$ using the notation of Lemma \ref{lmm:infinity}. Hence by Lemma \ref{lmm:infinity}, for $n$ large enough, there is a purely $\frac{1}{p}$-adically periodic element of $\Lambda$ beginning with $s_n.$ As $l$ can be taken as large as one wishes, this implies that the sequence $(m_1, m_2, \dots)$ is purely periodic, that is that $\lambda$ is purely $\frac{1}{p}$-adically periodic. This is absurd since $\lambda$ is in the complement of $\mathbb{Z}_{(p)},$ by hypothesis. We have thus shown the required vanishing.\qedhere
\end{proof}

This allows us to give an inductive system for the $D_R$-module $N_f$ expressed only in terms of the $F$-jumping exponents of $f$ in $(0, 1]\cap \mathbb{Z}_{(p)}.$  
Recall from Proposition \ref{prop:Frob pull-back} that for all $l$ large enough to separate the $F$-jumping exponents of $f,$ we have maps: $$(F^{l+1})^*(\mathfrak{gr}_{\lambda \in (0,1]}(\tau(f^{\lambda}))) \overset{a_{l+1}}{\to} (F^{l+2})^*(\mathfrak{gr}_{\lambda \in (0,1]}(\tau(f^{\lambda}))).$$ We set $$(F^{l+1})^*(\mathfrak{gr}_{\lambda \in (0,1] \cap \mathbb{Z}_{(p)}}(\tau(f^{\lambda}))) \overset{b_{l+1}}{\to} (F^{l+2})^*(\mathfrak{gr}_{\lambda \in (0,1] \cap \mathbb{Z}_{(p)}}(\tau(f^{\lambda})))$$ to be $b_{l+1}:= (F^{l+2})^*(\pi) \circ a_{l+1}\circ (F^{l+1})^*(i),$ where 
$i$ is the natural injection $$\mathfrak{gr}_{\lambda \in (0,1] \cap \mathbb{Z}_{(p)}}(\tau(f^{\lambda})) \overset{i}{\to} \mathfrak{gr}_{\lambda \in (0,1]}(\tau(f^{\lambda}))$$ and $\pi$ is the natural projection $\mathfrak{gr}_{\lambda \in (0,1]}(\tau(f^{\lambda})) \overset{\pi}{\to} \mathfrak{gr}_{\lambda \in (0,1] \cap \mathbb{Z}_{(p)}}(\tau(f^{\lambda})).$ We also have to consider the following intermediate system: $$(F^{l+1})^*(\mathfrak{gr}_{\lambda \in (0,1]}(\tau(f^{\lambda}))) \overset{c_{l+1}}{\to} (F^{l+2})^*(\mathfrak{gr}_{\lambda \in (0,1]}(\tau(f^{\lambda}))),$$ with $c_{l+1}:= (F^{l+2})^*(i) \circ (F^{l+2})^*(\pi)\circ a_{l+1}.$ We will denote the corresponding inductive systems by $(a_j), (b_j)$ and $(c_j),$ respectively.
  
\begin{cor} \label{cor: inductive system} On one hand, the injections $(F^j)^*(i)$ give rise to a morphism of inductive systems $(b_j) \to (c_j).$ Denote it $\bar{i}.$ On the other, the maps $$(F^{j})^*(\mathfrak{gr}_{\lambda \in (0,1]}(\tau(f^{\lambda}))) \overset{Id}{\to} (F^{j})^*(\mathfrak{gr}_{\lambda \in (0,1]}(\tau(f^{\lambda})))$$ and $$(F^{j+1})^*(\mathfrak{gr}_{\lambda \in (0,1]}(\tau(f^{\lambda}))) \overset{(F^{j+1})^*(i)\circ (F^{j+1})^*(\pi)}{\to} (F^{j+1})^*(\mathfrak{gr}_{\lambda \in (0,1]}(\tau(f^{\lambda})))$$ induce a morphism $(a_j) \to (c_j).$ Let us denote it $\bar{\pi}.$

The morphisms $\bar{i}$ and $\bar{\pi}$ induce isomorphisms of direct limits. In particular $N_f$ is isomorphic to the direct limit of $(b_j), \varinjlim(F^{j})^*(\mathfrak{gr}_{\lambda \in (0,1] \cap \mathbb{Z}_{(p)}}(\tau(f^{\lambda}))).$
\end{cor}

\begin{proof} It is a straightforward consequence of Proposition \ref{prop:structure}.\qedhere\end{proof}

We can now prove the
\begin{thm} \label{thm:structure N_f} Let $k$ be a perfect field of characteristic $p>0,R$ a smooth $k$-algebra and let $f$ be an element of $R.$
\begin{enumerate}

\item \label{item:unitF} Let $e$ be the lcm of the lengths of the periods of the $\frac{1}{p}$-adic expansions of the $F$-jumping exponents of $f$ in $(0,1] \cap \mathbb{Z}_{(p)}.$ 
The left $D_R$-module $N_f$ is a finitely generated unit $F^e$-module. 
 
\item \label{item:splitting}
For each $l$ separating the $F$-jumping exponents of $f,$ the unit $F^e$-module $N_f$ splits as a direct sum $$N_f=\bigoplus_{\lambda \in (0,1]\cap \mathbb{Z}_{(p)}} (N_f)_{\lambda},$$ where the $\lambda$ are the $F$-jumping exponents of $f$ in $(0,1]\cap \mathbb{Z}_{(p)}.$ Each summand is a direct limit $$(N_f)_{\lambda}:=\varinjlim_{j\geq0}(F^{l+1+ej})^*(\tau(f^{\lambda^{-}})/\tau(f^{\lambda}))$$ where the limit is for the maps induced by the system $(b_j)$ and it is non-trivial, $$(N_f)_{\lambda}\not=0.$$

\item \label{item:FrobEuler} 
The higher Euler operators act as endomorphisms of the unit $F^e$-module $N_f.$ 

\end{enumerate}

\end{thm}

\begin{proof} 
(\ref{item:unitF}) We use the inductive system $(b_j)$ of Corollary \ref{cor: inductive system}, which by that Corollary has $N_f$ as direct limit. Pick an $l$ that separates the $F$-jumping exponents of $f$ and denote $\beta_e$ the morphism $b_{l+e}\circ\dots\circ b_{l+2}\circ b_{l+1}:$
$$(F^{l+1})^*(\mathfrak{gr}_{\mu \in (0,1]\cap \mathbb{Z}_{(p)}}(\tau(f^{\mu}))) \overset{\beta_e}{\to} (F^e)^*(F^{l+1})^*(\mathfrak{gr}_{\mu \in (0,1]\cap\mathbb{Z}_{(p)}}(\tau(f^{\mu}))).$$ 
We claim that $\beta_e$ is a generator. It is enough to check that for all $r\geq 0,$ the pull-back $(F^{re})^*(\beta_e)$ coincides with the composition of the structure maps from $(b_j),$ that is  $(F^{re})^*(\beta_e)=b_{(r+1)e+l} \circ \dots\circ b_{re+l+1}.$ 
We see by the description of the structure maps from Proposition \ref{prop:Df}.\ref{item: maps}  and Proposition \ref{prop:F-jumpDf} that for all purely $\frac{1}{p}$-adically periodic $F$-jumping exponent $\lambda,$ both $\beta_{e, r}:= b_{(r+1)e+l} \circ \dots\circ b_{re+l+1}$ and $(F^{re})^*(\beta_e)$ send $$(F^{re+l+1})^*(\tau(f^{\lambda^{-}})/\tau(f^{\lambda})) \text{ to } 
(F^{(r+1)e+l+1})^*(\tau(f^{\lambda^{-}})/\tau(f^{\lambda})).$$ Indeed let $\frac{m}{p^{l+1}}$ be the truncated $\frac{1}{p}$-adic expansion of $\lambda.$ Since $e$ is a multiple of the periods, they both have to factor through the sum of the factors indexed by $F$-jumping exponents $\mu$ with the same truncated $\frac{1}{p}$-adic expansion $\frac{m}{p^{l+1}}.$ But $l$ separates the $F$-jumping exponents hence $\mu$ has to be equal to $\lambda.$ 

Moreover if $\frac{m}{p^{l+1}}+\frac{p^{l+1}n}{p^{e+l+1}}$ is the truncated $\frac{1}{p}$-adic expansion of $\lambda,$ we have by the description of the structure maps that $\beta_{e, r}$ is the multiplication by $cf^{p^{re+ l+1}n},$ for a non-zero element $c$ of $\mathbb{F}_p.$ We thus have that $\beta_{e, r}= (F^{re})^*(\beta_{e, 0}).$ Since $\beta_{e, 0}$ is equal to $\beta_e,$ this concludes the proof of the point.

(\ref{item:splitting}) The existence of the decomposition and the description of each summand as a direct limit follow directly from the description of $(b_j)$ in the proof of the point (\ref{item:unitF}) above. 

Let us prove  
the non-triviality of the summands, $(N_f)_{\lambda}\not=0.$ Fix an $F$-jumping exponent $\lambda$ of $f$ in $(0,1]\cap\mathbb{Z}_{(p)}.$ 
By Proposition \ref{prop:F-jumpDf} and the direct limit description of the summands,  
the vanishing $(N_f)_{\lambda}=0$ implies that the image under the structure map of $f^m \in D_R^{(l)}f^m$ in $D_R^{(l+je)}f^{m'}/D_R^{(l+je)}f^{m'+1}$ is zero, for some positive $j,$ and $m$ and $m'$ such that $\frac{m}{p^{l+1}}$ and $\frac{m'}{p^{l+je+1}}$ are the truncated $\frac{1}{p}$-adic expansions of $\lambda.$   
But by Proposition \ref{prop:Df}.\ref{item: maps}, 
the image of $f^m$ is a non-zero scalar times $f^{m'}.$ Thus its vanishing implies that of the quotient: $$0=D_R^{(l+je)}f^{m'}/D_R^{(l+je)}f^{m'+1}\cong (F^{l+je+1})^*(\tau(f^{\frac{m'}{p^{l+je+1}}})/\tau(f^{\frac{m'+1}{p^{l+je+1}}}))$$ and hence that of $$\tau(f^{\frac{m'}{p^{l+je+1}}})/\tau(f^{\frac{m'+1}{p^{l+je+1}}})=0.$$ Which by the definition of $m'$ is absurd since $\lambda$ is a $F$-jumping exponent of $f.$

(\ref{item:FrobEuler})
By (\ref{item:splitting}) 
it is enough to show the result for each of the unit $F^e$-modules $(N_f)_{\lambda}.$  
Moreover by definition of the isomorphisms in Corollary \ref{cor: inductive system}, the action of the higher Euler operators induced on the image of $(F^{l+1+ej})^*(\tau(f^{\lambda^{-}})/\tau(f^{\lambda}))$ in the direct limit of $(b_j)$ coincides with that induced by its image in the direct limit of $(a_j).$ It is thus given, via Proposition \ref{prop:F-jumpDf}, by Proposition \ref{prop:Df}.\ref{item: Euler}. 
That is the image of $(F^{l+1+ej})^*(\tau(f^{\lambda^{-}})/\tau(f^{\lambda}))$ is a common eigenspace of $\{\nu_0, \nu_1, \dots, \nu_{l+ej}\}$ of eigenvalue $-m_i$ for $\nu_i,$ where $\frac{\Sigma_{i=0}^{i=l+ej}m_ip^i}{p^{l+1+ej}}$ is the truncated $\frac{1}{p}$-adic expansion of $\lambda.$ Since $e$ is a multiple of the $\frac{1}{p}$-adic period of $\lambda,$ we have that if $m_{i+e}$ is defined, then $m_{i+e}=m_i.$ In particular for all numbers $i$ and $j'$ such that $0\leq i\leq l+j'e,$ the action of $\nu_i$ on the image of $(F^{l+1+ej'})^*(\tau(f^{\lambda^{-}})/\tau(f^{\lambda}))$ in $N_f$ is given by the same number $-m_i.$ The action of the higher Euler operators is thus compatible with the $F^e$-module structure.\qedhere
\end{proof}

\begin{rmk} \label{rmk:conjugates} The decomposition of $N_f$ as a direct sum $N_f=\bigoplus_{\lambda \in (0,1]\cap \mathbb{Z}_{(p)}} (N_f)_{\lambda}$ is canonical. However, the assignment of a $F$-jumping exponent to each of the summands is not. Indeed, different choices of $l$ in (\ref{item:splitting}) of Theorem \ref{thm:structure N_f} exchange the $(N_f)_{\lambda}$'s, for $\frac{1}{p}$-conjugated $\lambda$'s. This only depends on $l$ mod $e.$
\end{rmk}

We now briefly prove that the $b$-function has finitely many roots using the Riemann-Hilbert correspondence for unit $F$-modules (\cite{MR2071510}). This will be reproved in Theorem \ref{thm:periodic expansion of roots} which will provide more precise information about the roots. 

\begin{cor} \label{cor:finitely many roots} Let $k$ be a perfect field of characteristic $p>0,R$ a smooth $k$-algebra and let $f$ be an element of $R.$ 
The $b$-function of $f$ has finitely many roots.

\end{cor}

\begin{proof} We show that there are only finitely many maximal ideals of $$k[{s \choose p^0}, {s \choose p^1}, {s \choose p^2}, \cdots]$$ containing  the ideal $b_f \subset k[{s \choose p^0}, {s \choose p^1}, {s \choose p^2}, \cdots].$ It is an immediate consequence of Theorem \ref{thm:p-adic roots} that the maximal ideals of $k[{s \choose p^0}, {s \choose p^1}, {s \choose p^2}, \cdots]$ are defined over $\mathbb{F}_p.$ Thus it is enough to show that there are only finitely many maximal ideals of $\mathbb{F}_p[{s \choose p^0}, {s \choose p^1}, {s \choose p^2}, \cdots]$ containing $b_f\cap \mathbb{F}_p[{s \choose p^0}, {s \choose p^1}, {s \choose p^2}, \cdots].$ The latter is the annihilator of $N_f^{\gamma}$ in $\mathbb{F}_p[{s \choose p^0}, {s \choose p^1}, {s \choose p^2}, \cdots].$

By the Riemann-Hilbert correspondence for unit $F$-modules (\cite[Theorem 11.4.2]{MR2071510}), the unit $F^e$-module $N_f$ corresponds to an object in the constructible derived category of \'{e}tale $\mathbb{F}_p$-sheaves on $Spec(R).$ Since by (\ref{item:FrobEuler}) of Theorem \ref{thm:structure N_f}, the algebra $$\mathbb{F}_p[{s \choose p^0}, {s \choose p^1}, {s \choose p^2}, \cdots]$$ acts on $N_f$ by endomorphisms of unit $F^e$-module, it is transported by the Riemann-Hilbert correspondence to act on an object in the constructible derived category of \'{e}tale $\mathbb{F}_p$-sheaves on $Spec(R).$ The algebra of global endomorphims of those being finite dimensional over $\mathbb{F}_p,$ the algebra $\mathbb{F}_p[{s \choose p^0}, {s \choose p^1}, {s \choose p^2}, \cdots]/b_f$ is finite dimensional over $\mathbb{F}_p.$ It has thus finitely many maximal ideals, which proves the corollary.\qedhere
\end{proof}

\subsection{The roots of the $b$-function and $F$-jumping exponents}

Let us now describe the roots of the $b$-function of $f$ in terms of its $F$-jumping exponents.

\begin{thm} \label{thm:periodic expansion of roots} Let $k$ be a perfect field of characteristic $p>0,R$ a smooth $k$-algebra and let $f$ be an element of $R$ not contained in $k.1.$
The roots of the $b$-function of $f$ are the opposites of the $F$-jumping exponents of $f$ which are in $(0,1] \cap \mathbb{Z}_{(p)}.$
\end{thm}

\begin{proof} By Theorem \ref{thm:structure N_f}, $$b_f:=ann_{k[{s \choose p^0}, {s \choose p^1}, {s \choose p^2}, \cdots]}(N_f^{\gamma})=\cap_{\lambda \in (0,1]\cap \mathbb{Z}_{(p)}}ann_{k[{s \choose p^0}, {s \choose p^1}, {s \choose p^2}, \cdots]}(N_f)_{\lambda}^{\gamma},$$ where the $\lambda$ are the $F$-jumping exponents of $f$ in $(0,1]\cap \mathbb{Z}_{(p)}$ and the exponent $\gamma$ is a reminder that $k[{s \choose p^0}, {s \choose p^1}, {s \choose p^2}, \cdots]$ acts via the isomorphism $\gamma$ of Lemma \ref{lmm: gamma}.

Let $\lambda$ be an $F$-jumping exponent of $f$ in $(0,1] \cap \mathbb{Z}_{(p)}.$ Recall that the relation between the summand $(N_f)_{\lambda}$ of $N_f$ and the $F$-jumping exponent $\lambda$ depends on the choice of an integer $l$ separating the $F$-jumping exponents of $f.$ For all $l,$ we described the action of the higher Euler operators on $(N_f)_{\lambda}$ in the proof of Theorem \ref{thm:structure N_f}.\ref{item:FrobEuler}. Namely the action of ${s \choose p^i}$ on $(N_f)_{\lambda}$ is the multiplication by $m_i,$ where $\frac{\Sigma_{r=0}^{r=l+ej}m_rp^r}{p^{l+1+ej}}$ is the truncated $\frac{1}{p}$-adic expansion of $\lambda,$ independently of $e.$ Thus if the $\frac{1}{p}$-adic expansion of $\lambda$ is $\lambda=\Sigma_{n\geq1}\frac{\lambda_n}{p^n},$ the annihilator of $(N_f)_{\lambda}^\gamma$ is the maximal ideal corresponding by Theorem \ref{thm:p-adic roots} to the $p$-adic integer $\Sigma_{i\geq0}m_ip^i,$ with $m_i= \lambda_{i'}$ for $i'= l+1-i\text{ mod }e.$ The latter is well-defined since $m_{i+e}=m_i$ and $\lambda_{j+e}=\lambda_j,$ for all $i,j.$ 

Hence for a choice of $l$ such that $l+1$ is a multiple of $e,$ say $l+1=re,$ we have that the annihilator of $(N_f)_{\lambda}^\gamma$ is the maximal ideal corresponding to the $p$-adic integer $\Sigma_{i\geq0}\lambda_{i'}p^i,$ where $i'= re-i\text{ mod }e.$ Which by Lemma \ref{lmm:expansion of opposite} is equal to $\Sigma_{i\geq0}\lambda_{i'}p^i= -\lambda.$

Since we have in particular shown that the set of roots of the $b$-function of $f$ is the union for all $F$-jumping exponents $\mu$ of $f$ in $(0,1] \cap \mathbb{Z}_{(p)}$ of the $p$-adic integers corresponding to the annihilator of $(N_f)_{\mu}^\gamma,$ this concludes the proof of the theorem.\qedhere 

\end{proof}

\begin{rmk} Since there are only finitely many $F$-jumping exponents of $f$ in $(0,1]$ by Theorem \ref{thm:F-jumps}, Theorem \ref{thm:periodic expansion of roots} implies Corollary \ref{cor:finitely many roots}, as announced.
\end{rmk}

The following is a direct consequence of Theorem \ref{thm:periodic expansion of roots}. 

\begin{cor} \label{cor:negratroots} The roots of the $b$-function are negative rational numbers, $\geq -1.$
\end{cor}

\section{Examples} \label{section: ex}

In this section, we use Theorem \ref{thm:periodic expansion of roots} to compute the roots of the $b$-function in some examples.

\begin{ex} Let $f=x$ in $\mathbb{F}_p[x].$ The set of roots of the $b$-function of $f$ is $\{-1\}.$ Note that the Bernstein-Sato polynomial of $x$ in $\mathbb{C}[x]$ is $(s+1).$
\end{ex}

\begin{ex} Let $f=x_1^2+\dots+x_n^2$ in $\mathbb{F}_p[x_1, \dots, x_n],$ where $n\geq2$ and $p>2.$ The only $F$-jumping exponent of $f$ in $(0,1]$ is 1 (\cite[Example 6.16]{MR2469353}). Thus by Theorem \ref{thm:periodic expansion of roots}, the set of roots of the $b$-function of $f$ is $\{-1\}.$ Note that the Bernstein-Sato polynomial of $x_1^2+\dots+x_n^2$ in $\mathbb{C}[x_1, \dots, x_n]$ is $(s+\frac{n}{2})(s+1)$ (\cite[Example 6.2]{MR1943036}).\end{ex}

\begin{ex} 

For all $n$ and $j, 1\leq j\leq n,$ let $\alpha_j$ be a positive integer.  
It is well-known (\cite[Theorem 6.10]{MR1974679}) that the $F$-jumping exponents of the element $f=x_1^{\alpha_1}\dots x_n^{\alpha_n}$ of $\mathbb{F}_p[x_1,\dots,x_n]$ in $(0,1]$ are $\bigcup_{1\leq j\leq n} \Big\{\frac{l}{\alpha_j}|1\leq l\leq \alpha_j\Big\}.$
Thus by Theorem \ref{thm:periodic expansion of roots}, the set of roots of the b-function of $f$ is $$\bigcup_{1\leq j\leq n} \Big\{-\frac{l}{\alpha_j}|1\leq l\leq \alpha_j\Big\}\cap \mathbb{Z}_{(p)}.$$

Note that by \cite[Lemma 6.10]{MR1943036}, the Bernstein-Sato polynomial of the complex polynomial $x_1^{\alpha_1}\dots x_n^{\alpha_n}$  
is $\Pi_{1\leq j\leq n} \Pi_{1\leq l\leq \alpha_j} (s+\frac{l}{\alpha_j}).$
\end{ex}

\begin{ex} 
Suppose that $p$ is at least $5$ and let $f_p=x^2+y^3$ in $\mathbb{F}_p[x,y].$  
By \cite[Example 6.14]{MR2469353}, the $F$-jumping exponents of $f_p$ in $(0,1]$ are: $\{\frac{5}{6}, 1\} \textit{ if } p=1 \text{ mod } 3$ and $\{\frac{5}{6}-\frac{1}{6p}, 1\} \textit{ if } p=2\text{ mod } 3.$
Hence by Theorem \ref{thm:periodic expansion of roots}, the roots of $b_{f_p}$ are:

$$\Big\{-1, -\frac{5}{6}\Big\} \textit{ if } p=1\text{ mod } 3$$ and
$$\Big\{-1\Big\} \textit{ if } p=2\text{ mod } 3.$$

Note that by \cite[Example 6.19]{MR1943036}, the Bernstein-Sato polynomial of $f=x^2+y^3$  
is $$b_f=(s+\frac{7}{6})(s+1)(s+\frac{5}{6}).$$

\end{ex}

\begin{ex} \label{ex: Takagi}

Let $f_p=x^5+y^4+x^3y^2$ in  $\mathbb{F}_p[x,y].$ By \cite[Example 4.5]{MR2185754}, if $p=19\text{ mod } 20,$ then $\frac{9p-11}{20(p-1)}$ is a $F$-jumping exponent of $f_p.$ 
Thus Theorem \ref{thm:periodic expansion of roots} implies that if $p=19\text{ mod } 20,$ then 
$$-\frac{9p-11}{20(p-1)}$$ is a root of the $b$-function of $f_p.$
\end{ex}

\bibliographystyle{plain}
\bibliography{bibfilex}

\end{document}